\documentclass[12pt]{amsart}

\usepackage{setspace}
\usepackage{amsmath, amssymb,amsthm, amsfonts, latexsym, mathrsfs}
\usepackage{hyperref}
\usepackage{color}
\usepackage{graphicx}
\usepackage[normalem]{ulem} 

\newtheorem{theo}{Theorem}[section]
\newtheorem{conj}{Conjecture}[section]
\newtheorem{formula}{Formula}[section]
\newtheorem*{thm*}{Theorem}

\newtheorem*{prop*}{Proposition}

\theoremstyle{remark}
\newtheorem{rem}{Remark}[section]

\renewcommand{\div}{\mbox{div}}

\newcommand{\Ric}{\mbox{Ric}}
\newcommand{\R}{\mathbb R}

\numberwithin{equation}{section}
\newcommand{\be}{\begin{equation}}
\newcommand{\ee}{\end{equation}}
\def\p{\partial}

\def\la{\langle}
\def\ra{\rangle}

\def\Pi{\displaystyle{\mathbb{II}}}
\def\Ric{\text{\rm Ric}}

\def\vh{\vspace{.2cm}}

\def\m{\mathfrak{m}}

\def\bg{\bar{g}}

\def\bee{\begin{equation*}}
\def\eee{\end{equation*}}

\def\vh{\vspace{.3cm}}

\def\So{\Sigma_{_{O}}}
\def\Sh{\Sigma_{H}}

\def\bg{\bar g}
\def\bSigma{\bar \Sigma}
\def\bF{\bar F}
\def\bH{\bar H}

\def\bnu{\bar \nu} 
\def\bA{\bar A}

\def\mb{\mathfrak{m}_{_B}}

\def\R{\mathbb{R}}
\def\StM{\mathbb{M}^{3}_m}

\def\Ric{\mathrm{Ric}}
\def\bg{\bar g}

\def\bRic{\overline{\Ric}}  
\def\So{\Sigma_{_O}}      
\def\Sh{\Sigma_{_H}}      
\def\fg{\breve{g}}              
\def\fH{{H}}                       
\def\m{\mathfrak{m}}        
\def\N{\mathbb{N}}           
\def\mB{\m_{_B}}

\def\tM{\tilde M}
\def\tg{\tilde g}
\def\S{\mathcal{S}}
\def\ESwy{E^S_{_{WY}}}

\newcounter{mnotecount}[section]

\begin{document}

\title{Variation and rigidity of quasi-local mass}

\author[Siyuan Lu]{Siyuan Lu}
\address{Department of Mathematics, Rutgers University, 110 Frelinghuysen Road, Piscataway, NJ 08854, USA.}
\email{siyuan.lu@math.rutgers.edu}

\author[Pengzi Miao]{Pengzi Miao}
\address{Department of Mathematics, University of Miami, Coral Gables, FL 33146, USA.}
\email{pengzim@math.miami.edu}

\thanks{The first named author's research was partially supported by  CSC fellowship. The second named author's research was partially supported by the Simons Foundation Collaboration Grant for Mathematicians \#281105.}

\begin{abstract}
Inspired by the work of Chen-Zhang \cite{Chen-Zhang}, 
we derive an evolution formula for the Wang-Yau quasi-local energy in reference to 
a static space, introduced by Chen-Wang-Wang-Yau \cite{CWWY}.
If the reference static space represents a mass minimizing, static extension of the initial surface $\Sigma$, we observe that the derivative of the Wang-Yau quasi-local energy is equal to the derivative of 
the Bartnik quasi-local mass at $\Sigma$.  

Combining the evolution formula for the quasi-local energy with a localized  Penrose inequality proved in \cite{Lu-Miao}, 
we prove a rigidity theorem for compact $3$-manifolds with nonnegative scalar curvature, with boundary. 
This rigidity theorem in turn gives a characterization of the equality case of the localized Penrose inequality in $3$-dimension. 

\end{abstract}

\maketitle 

\markboth{Siyuan Lu and Pengzi Miao}{Variation and rigidity of quasi-local mass}

\section{Introduction}

The purpose in this paper is twofold. 
We derive a derivative formula for the integral 
\be \label{eq-NH-integral}
 \int_{\Sigma_t} N (\bH - H ) \, d \sigma 
\ee 
along a family of hypersurfaces $ \{ \Sigma_t \} $ evolving in a Riemannian manifold $(M, g)$
with an assumption that $ \Sigma_t $ can be isometrically embedded in a static space $(\N, \bg)$ 
as a comparison hypersurface $\bSigma_t$. Here $ H$, $\bH$ are the mean curvature of $ \Sigma_t$, 
$ \bSigma_t$ in $(M, g)$, $(\N, \bg)$, respectively, and $N$ is the static potential  on $(\N, \bg)$. 
When $ \{ \Sigma_t \}$ is a family of closed $2$-surfaces in a $3$-manifold $(M,g)$, 
integral  \eqref{eq-NH-integral} represents 
the Wang-Yau quasi-local energy in reference to the static space $(\N, \bg)$, 
introduced by Chen-Wang-Wang-Yau \cite{CWWY}.
In this case, if $(\N, \bg)$ represents a mass minimizing, static extension of the initial surface $\Sigma_0$,
we find that  the derivative of the quasi-local energy  
agrees with the derivative of the Bartnik quasi-local mass at $\Sigma_0$ 
(see \eqref{eq-evolution-qlnergy-mb} in Section \ref{sec-evolution-qlmass}). 

We also apply the derivative formula of \eqref{eq-NH-integral} 
to prove a rigidity theorem for compact $3$-manifolds with nonnegative scalar curvature, with boundary. 
Precisely, we have

\begin{theo}  \label{thm-main}
Let  $(\Omega, \fg)$ be a compact, connected,  orientable, $3$-dimensional Riemannian manifold 
with nonnegative scalar curvature, with boundary $\p \Omega$. 
Suppose $ \p \Omega$ is the disjoint union of two pieces, $ \So $ and $ \Sh$, where 
\begin{enumerate}
\item[(i)]   $\So$ has positive mean curvature $\fH$; and  
\item[(ii)]  $\Sh$ is a minimal hypersurface
(with one or more components)
and there are no other closed  minimal hypersurfaces in $ (\Omega, \fg)$.  
\end{enumerate}
Let $ \StM$ be a  $3$-dimensional spatial Schwarzschild manifold with  mass $ m >0 $
outside the horizon.
Suppose $ \So$ is isometric to a convex surface $\Sigma \subset \StM  $ which 
encloses  a   domain $\Omega_m $ with  the  horizon $ \p \StM $.
Suppose $ \bRic(\nu, \nu) \le 0 $ on $ \Sigma$, where ${\bRic}$ is the  Ricci curvature of 
the Schwarzschild metric $\bg$ on  $\StM$ and $ \nu $ is the outward unit normal to $ \Sigma$. 
Let $H_m  $ be  the mean curvature of $ \Sigma$ in $ \StM$ and 
$ | \Sh |$ be  the area of $\Sh $ in $(\Omega, \fg)$.
If  $ H =  H_m$ and   $\sqrt{ \frac{ | \Sh | }{ 16 \pi}  } = m$,
then $(\Omega, \fg)$ is isometric to $ (\Omega_m, \bg)$.
\end{theo}

Theorem \ref{thm-main} gives a characterization of the equality case of 
a localized Penrose  inequality proved in \cite{Lu-Miao}.

\begin{theo} [\cite{Lu-Miao}]  \label{thm-L-M}
Let  $(\Omega, \fg)$ be a compact, connected,  orientable, $3$-dimensional Riemannian manifold 
with nonnegative scalar curvature, with boundary $\p \Omega$. 
Suppose $ \p \Omega$ is the disjoint union of two pieces, $ \So $ and $ \Sh$, where 
\begin{enumerate}
\item[(i)]   $\So$ has positive mean curvature $\fH$; and  
\item[(ii)]  $\Sh$, if nonempty,  is a minimal hypersurface
(with one or more components)
and there are no other closed  minimal hypersurfaces in $ (\Omega, \fg)$.  
\end{enumerate}
Let $ \StM$ be a  $3$-dimensional spatial Schwarzschild manifold with  mass $ m >0 $
outside the horizon.
Suppose $ \So$ is isometric to a convex surface $\Sigma \subset \StM  $ which 
encloses  a  domain $\Omega_m $ with  the  horizon $ \p \StM $.
Suppose $ \bRic(\nu, \nu) \le 0 $ on $ \Sigma$, where ${\bRic}$ is the  Ricci curvature of 
the Schwarzschild metric $\bg$ on  $\StM$ and $ \nu $ is the outward unit normal to $ \Sigma$. 
Then
\be \label{ineq-main}
m + \frac{1}{ 8 \pi } \int_\Sigma N ( H_m  - H )   \, d \sigma  \geq 
 \sqrt{ \frac{ | \Sh | }{ 16 \pi}  } .
\ee
Here $N$ is the static potential on $ \StM$, 
$H_m$ is the mean curvature of $ \Sigma$ in $ \StM$,  and 
$ | \Sh |$ is the area of $\Sh $ in $(\Omega, \fg)$. 
Furthermore, equality in \eqref{ineq-main} holds if and only if  
\be \label{eq-H-complete}
  H =  H_m, \  \sqrt{ \frac{ | \Sh | }{ 16 \pi}  } = m .
\ee
\end{theo}

By Theorems \ref{thm-main} and \eqref{eq-H-complete}, we have the following rigidity statement concerning 
the equality case of \eqref{ineq-main}.

\begin{theo} 
Equality in \eqref{ineq-main} in Theorem \ref{thm-L-M} holds if and only if $(\Omega, \fg)$ is 
isometric to $(\Omega_m, \bg)$.
\end{theo}

Our motivation to consider the evolution of \eqref{eq-NH-integral} and 
the  proof of Theorem \ref{thm-main} are  inspired by 
a recent paper 
of Chen and Zhang \cite{Chen-Zhang}.
In \cite{Chen-Zhang}, Chen-Zhang proved the global rigidity of a convex surface $\Sigma$ 
with $ \bRic(\nu, \nu) \le 0 $ among all isometric surfaces $ \Sigma'$ in $\StM$ having the same mean 
curvature and enclosing the horizon. 
As a key step in their proof,
they computed  the first variation of 
the quasi-local energy of $ \Sigma'$ in reference to $\StM$. 
Such a variational consideration is made possible by  the openness result  of solutions to the isometric embedding problem
 into warped product space, which is due to  Li and Wang \cite{Li-Wang}.
Combining the  variation formula with inequality \eqref{ineq-main}, Chen-Zhang
established the rigidity of $ \Sigma$ in $\StM$.

This paper may be viewed as a further application of the method of Chen-Zhang.
In Section \ref{sec-evolution-qlmass}, we compute the derivative of \eqref{eq-NH-integral} 
(see Formula \ref{formula-I})
and relate it to the derivative of the Bartnik quasi-local mass.
In Section \ref{sec-rigidity}, we prove Theorem \ref{thm-main} by applying Formula \ref{formula-I} 
and Theorem \ref{thm-L-M}. 
In Section \ref{sec-discussion}, we discuss the implication of \eqref{eq-evolution-qlnergy-mb}
 on the relation between the Bartnik mass and the  Wang-Yau quasi-local energy.

\section{Evolution of quasi-local mass} \label{sec-evolution-qlmass}

In this section we derive a formula that is inspired by \cite[Lemma 2]{Chen-Zhang}.
First we fix some notations.
Let $(M, g)$ be an $(n+1)$-dimensional Riemannian manifold and 
 $ \Sigma$ be an $n$-dimensional closed manifold. 
Consider a family of embedded  hypersurfaces  $\{ \Sigma_t \}$  
evolving in $(M, g)$ according to 
$$
F : \Sigma \times I \longrightarrow  M, \ \ 
 \displaystyle  \frac{\p F}{\p t} = \eta \nu . $$ 
Here $F$ is a smooth map,  $I$ is some open interval containing $0$, $ \Sigma_t = F_t (\Sigma)$ with $F_t (\cdot) = F (\cdot, t)$, 
$ \nu$ is a chosen unit normal to $ \Sigma_t  = F_t (\Sigma)$, and $ \eta $ denotes the speed of the evolution of $\{ \Sigma_t \}$. 

Let $(\N, \bg)$ denote an $(n+1)$-dimensional {\em static} Riemannian manifold.
Here  $(\N,\bar{g})$ is called static (cf. \cite{Corvino}) if there exists a nontrivial function $N$ such that
\begin{align}\label{eq-static}
(\bar{\Delta}N)\bar{g}-\bar{D}^2N+N\bar{Ric}=0,
\end{align}
where $\bar{Ric}$ is the Ricci curvature of $(\N, \bar g)$, $\bar{D}^2 N $ is the Hessian of $N$ 
and $\bar{\Delta}$ is the Laplacian of $N$. 
The function $N$ is called a static potential on $(\N, \bg)$.

In what follows, we  consider another family of embedded hypersurfaces  $\{ \bSigma_t \}$  
evolving in $(\N, \bg)$ according to 
$$ \bF :  \Sigma \times I  \longrightarrow \N $$
with $ \bSigma_t = \bF_t (\Sigma)$ and $\bF_t (\cdot) = \bF (\cdot, t)$. 
We will make an important assumption:
\be \label{eq-iso}
 \bF_t^* (\bar g) = F_t^* (g ) , \ \forall \ t \in I. 
\ee
In particular, this means that $ \bSigma_t $ is assumed to be isometric to $ \Sigma_t$ for each $t$.

\begin{rem}
We  emphasize that, when $ n = 2$, given any $\{ \Sigma_t \}$ in  $(M, g)$,
if  $ \Sigma_0$ admits an isometric embedding into $(\N, \bg)$, 
there exists a family of $ \{ \bSigma_t \}$ in  $(\N, \bg)$ 
satisfying condition \eqref{eq-iso}. This is guaranteed by the openness result of solutions to the isometric embedding problem, which is due to
Li and Wang \cite{Li-Wang, Li-Wang-more}.
\end{rem}

We  will compute 
$$ \displaystyle 
\frac{d}{dt} \int_\Sigma N_t (\bH_t - H_t) \, d \sigma_t ,$$ 
where $ N_t = \bF_t^*(N) $ is the pull back of the static potential  $N$ on $(\N, \bar g)$;
$ H_t$, $\bH_t$ are the mean curvature of $ \Sigma_t$, $\bSigma_t$ in $(M, g)$, $(\N, \bar g)$,
respectively;
and $ d \sigma_t $ is the area element of the pull back metric $\gamma_t = \bF_t^* (\bar g) = F_t^* (g)$.
For simplicity,  the lower index $t$ is omitted below.  

\vh

\begin{formula} \label{formula-I}
Given $\{ \Sigma_t\}$, $\{ \bSigma_t \}$ evolving in $(M,g)$, $ (\N, \bg)$ as specified above, 
\be \label{eq-evolution-qlmass}
\begin{split}
& \ \frac{d}{dt} \int_\Sigma N (\bH - H ) \, d \sigma \\
= & \ \int_\Sigma N \left[ \frac12 | A - \bA |^2 - \frac12 | H -  \bH |^2  + \frac12 ( R - \bar R )  \right] \eta \, d \sigma \\
& \ + \int_\Sigma   \left[ ( f  - \eta )      \frac{\p N}{\p \bnu}  + \la \nabla N, Y \ra \right]  ( \bH - H )   \, d \sigma  .
\end{split}
\ee
Here $A$, $\bA$ are the second fundamental forms of $\Sigma_t$, $\bSigma_t$ in $(M, g)$, $ (\N, \bg)$, respectively;
$ R$, $ \bar R$ are the scalar curvature of $(M, g)$, $ (\N, \bg)$, respectively; 
 $f$ and $ Y$ are the lapse and the shift associated to $ \displaystyle \frac{\p \bF}{\p t}$, i.e. 
$ \displaystyle \frac{\p \bF}{\p t} =   f \bnu + Y$, 
where $f$ is a function and $Y $ is  tangential to $\bSigma_t$; and 
$ \nabla $ denotes the gradient on $(\bSigma_t, \gamma)$.
\end{formula}

\begin{rem}
Suppose $(M, g)$ and $(\N, \bg)$ both are $\StM$  and suppose 
 $ H = \bH $ at $t=0$, \eqref{eq-evolution-qlmass} becomes  
\be \label{eq-Chen-Zhang}
\begin{split}
 \frac{d}{dt} |_{t=0}  \int_\Sigma N (\bH - H ) \, d \sigma 
=  \frac12 \int_\Sigma N  | A - \bA |^2    \eta \, d \sigma .
\end{split}
\ee
This is the formula in \cite[Lemma 2]{Chen-Zhang}.
\end{rem}

\begin{rem}
If $Y=0 $ and $R = \bar R $, \eqref{eq-evolution-qlmass} reduces to 
\be
\begin{split}
& \  \frac{d}{dt} \int_\Sigma N (\bH - H ) \, d \sigma \\
= & \ \int_\Sigma N \left[ \frac12 | A - \bA |^2 - \frac12 | H -  \bH |^2   \right] \eta \, d \sigma 
 + \int_\Sigma  ( f  - \eta )      \frac{\p N}{\p \bnu} ( \bH - H )   \, d \sigma  \\
= & \ \int_\Sigma  \eta^{-1}  ( f - \eta )^2 \left( - N \sigma_2  -  \bH    \frac{\p N}{\p \bnu} \right)      \, d \sigma.
\end{split}
\ee
This is the formula in \cite[Proposition 2.2]{Lu-Miao}. 
\end{rem}

We now comment on the physical meaning of \eqref{eq-evolution-qlmass}.
Suppose $n = 2$. 
In \cite{CWWY},  Chen, Wang, Wang and Yau introduced a notion 
of quasi-local energy of a $2$-surface $\Sigma$ 
in reference to the static spacetime $ \S  = ( \R^{1} \times \N , - N^2 d t^2 + \bg ) $.
The notion is a generalization of  the Wang-Yau quasi-local energy \cite{WY1, WY2} for which the reference 
$\S$  is the Minkowski spacetime $\R^{3,1}$. 
For this reason, 
we denote this quasi-local energy of $ \Sigma$ by 
$ \ESwy ( \Sigma, \S, X) $, where $ X: \Sigma \rightarrow \S $ is an associated isometric embedding.
When $ \Sigma$ lies in a time-symmetric slice in the physical spacetime, 
 one may focus on  the case $X$ embeds $\Sigma$ into  a constant $t$-slice of $\S$, i.e.
 $ X : \Sigma \rightarrow (\N, \bg)$. 
In this case, setting $ \tau = 0 $ in equation (2.10) in \cite{CWWY}, 
one has 
\be \label{eq-ESWY}
\ESwy(\Sigma, \S, X) =  \frac{1}{8 \pi } \int_{\Sigma} N (\bH - H)  \, d \sigma .
\ee
Therefore, up to a multiplicative constant, \eqref{eq-evolution-qlmass} 
is a formula of 
$$  \frac{d}{dt} \ESwy (\Sigma_t, \S, X_t) ,$$
where $ X_t = \bar F_t \circ  F_t^{-1} $ is the isometric embedding of 
$\Sigma_t$ in $(\N, \bg)$ as $\bar \Sigma_t$. 

Next, we tie \eqref{eq-evolution-qlmass} with the evolution formula of the Bartnik quasi-local mass 
$\mb (\cdot) $.
We defer the detailed definition of the Bartnik mass $\mb(\cdot) $ to Section \ref{sec-discussion}.
For the moment,  we recall the following evolution formula of $\mb (\cdot)$ 
derived in \cite[Theorem 3.1]{Miao-ICCM-07} under a stringent condition. 

\begin{formula}  [\cite{Miao-ICCM-07}] \label{formula-bm}
Suppose $\Sigma_t $ has a mass minimizing, static extension $(M_t^s, g_t^s)$
such that $\{ (M_t^s, g_t^s )\}$ depends smoothly on $t$. One has 
\be \label{eq-evolution-bm}
\frac{d}{dt}|_{t=0}  \mathfrak{m}_{_B} (\Sigma_t) = 
 \frac{1}{16 \pi} \int_\Sigma N \left( | A - \bA |^2   +  R   \right) \eta \, d \sigma .
\ee
\end{formula}

To relate \eqref{eq-evolution-qlmass} to \eqref{eq-evolution-bm}, 
we assume that  $(\N, \bg)$ represents  a  mass minimizing, 
static extension of  the surface $\Sigma_0 \subset (M, g)$. Then, by assumption,
$ H = \bH $ at $t=0$. 
It  follows from \eqref{eq-evolution-qlmass}, \eqref{eq-ESWY} and \eqref{eq-evolution-bm} that 
\be \label{eq-evolution-qlnergy-mb}
\begin{split}
& \  \frac{d}{dt}|_{t=0} 
\ESwy(\Sigma_t, \S, X_t) \\ 
= &  \frac{1}{16 \pi} \int_\Sigma N \left[  | A - \bA |^2  +  ( R - \bar R )  \right] \eta \, d \sigma \\
=  & \ \frac{d}{dt}|_{t=0}  \mathfrak{m}_{_B} (\Sigma_t) .
\end{split}
\ee
We will reflect more on this relation in Section \ref{sec-discussion}.

In the remainder of this section, we give a proof of Formula \ref{formula-I}.

\begin{proof}[Proof of Formula \ref{formula-I}]
By the evolution equations
$ \frac{\p F}{\p t} = \eta \nu$ and $ \frac{\p \bF}{\p t} = f \bnu + Y $, we have
\be
\gamma' = 2 \eta A ,  \ 
\p_t d \sigma  = \eta H \,  d \sigma
\ee
and
\be \label{eq-metric-evolution}
\gamma' = 2 f \bA + L_Y \gamma , \ 
\p_t d \sigma = ( f \bH + \div Y ) \, d \sigma ,
\ee
where  $ \div  Y $ is  the divergence of $ Y $ on $(\Sigma, \gamma)$.
Thus, 
\be 
2 \eta A = 2 f \bA + L_Y \gamma  , \ \ 
\eta H = f \bH + \div Y \, . 
\ee

\vh

We first compute 
\be
\begin{split}
\frac{d}{d t} \int_\Sigma N \bH  \, d \sigma  = & \ 
\int_\Sigma  ( N' \bH  + N \bH' )  \    d \sigma   + N \bH \, \p_t d \sigma .
\end{split}
\ee
Let $ \bar \nabla $  denote the gradient  on $(\N, \bg)$. We have
\be
\begin{split}
N' = \la \bar \nabla N , \frac{\p \bF}{\p t} \ra =
\la \bar \nabla N , f \bnu + Y  \ra 
= f \frac{\p N}{\p \bnu} + \la \nabla N, Y \ra . 
\end{split}
\ee
Hence,
\be \label{eq-I-term}
\begin{split}
 \int_\Sigma N'  \bH \, d \sigma = \int_\Sigma  \left(  f \frac{\p N}{\p \bnu} + \la \nabla N, Y \ra \right)  \bH
 \, d \sigma . 
\end{split}
\ee
Recall that
\be
\bA'_{\alpha \beta}  = f \bA_{\alpha \delta} \bA^\delta_{\beta}  + ( L_Y \bA )_{\alpha \beta} 
- (\nabla^2 f)_{\alpha \beta}  + f \la \bar R (\bnu, \p_\alpha )\bnu, \p_\beta \ra ,
\ee
where $ \nabla^2 $ denotes  the Hessian on $(\Sigma, \gamma)$.
Hence,
\be
\begin{split}
\bH' = & \ (\gamma^{\alpha \beta})' \bA_{\alpha \beta} + \gamma^{\alpha \beta} \bA'_{\alpha \beta}  \\
= & \ - \la \gamma', \bA \ra + f |\bA|^2 +  \la \gamma, L_Y \bA \ra - \Delta f - \bar Ric (\bnu, \bnu) f . 
\end{split}
\ee
By \eqref{eq-metric-evolution},
\bee
\begin{split}
\la \gamma' , \bA \ra =  \la 2 f \bA + L_Y  \gamma , \bA \ra =  2 f | \bA |^2 + \la L_Y  \gamma , \bA \ra .
\end{split}
\eee
Thus, 
\be
\begin{split}
\bH' = & \ - \la L_Y  \gamma , \bA \ra  +  \la \gamma, L_Y  \bA \ra - \Delta f -  f | \bA |^2  - \bar Ric (\bnu, \bnu) f . 
\end{split}
\ee
One checks that 
\be
- \la L_Y  \gamma , \bA \ra  +  \la \gamma, L_Y  \bA \ra =  \la Y , \nabla \bH \ra . 
\ee
Hence, 
\be
\bH' =  - \Delta f -  f | \bA |^2  - \bar Ric (\bnu, \bnu) f + \la Y , \nabla \bH \ra .
\ee
Thus, 
\be \label{eq-II-term}
\begin{split}
 \int_\Sigma N \bH' \, d \sigma 
= & \ \int_\Sigma
( -  \Delta N  - \bar Ric (\bnu, \bnu) N )  f +  N \left[  -  f | \bA |^2   + \la Y , \nabla \bH \ra \right]  \, d \sigma \\
= & \ \int_\Sigma \bH \frac{\p N}{\p \bnu}    f  -   N f | \bA |^2   +  N \la Y , \nabla \bH \ra  \, d \sigma  .
\end{split} 
\ee
Here we have used 
\bee
 \Delta N  +  \bar Ric (\bnu, \bnu) N =  - \bH \frac{\p N}{\p \bnu}   ,
\eee
which follows from the static equation \eqref{eq-static}.

By \eqref{eq-I-term} and \eqref{eq-II-term}, 
\be \label{eq-2nd-term}
\begin{split}
& \ \int_\Sigma  N'  \bH + N \bH' \, d \sigma \\
= & \  \int_\Sigma  \left(  f \frac{\p N}{\p \bnu} + \la \nabla N, Y  \ra \right)  \bH
+  \bH \frac{\p N}{\p \bnu}    f  -    N f  | \bA |^2   +  N \la Y , \nabla \bH \ra  \, d \sigma \\
= & \ \int_\Sigma  2 f \frac{\p N}{\p \bnu} \bH  
-    N f  | \bA |^2   +   \la Y , \nabla ( N \bH ) \ra  \, d \sigma . 
\end{split}
\ee
On the other hand, by  \eqref{eq-metric-evolution},
\be \label{eq-III-term}
\int_{\Sigma} N \bH \, \p_t d \sigma = \int_{\Sigma} N \bH  (f \bH + \div Y ) \,  d \sigma .
\ee
Therefore, it follows from \eqref{eq-2nd-term} and \eqref{eq-III-term} that
\be \label{eq-over-I}
 \frac{d}{dt} \int_\Sigma N \bH \, d \sigma 
=   \int_\Sigma 2  f \frac{\p N}{\p \bnu} \bH    +  N f  ( \bH^2 - | \bA |^2 )  \, d \sigma  \, . 
\ee

To proceed, we note that by \eqref{eq-metric-evolution}, 
\be
\begin{split}
 2  f  ( \bH^2 - | \bA |^2 ) =   \la \bH \gamma  - \bA , 2 f \bA \ra  
=   \la \bH \gamma  - \bA , \gamma' \ra - \la \bH \gamma - \bA , L_Y  \gamma  \ra  \, .
\end{split} 
\ee
Thus, 
\be
\begin{split}
2 \int_\Sigma N   f  ( \bH^2 - | \bA |^2 ) \, d \sigma   = & \ 
\int_\Sigma N  \la \bH \gamma  - \bA , \gamma' \ra -    N \la \bH \gamma - \bA , L_Y  \gamma  \ra 
\, d \sigma . 
\end{split}
\ee
Integrating by parts, we have 
\be
\begin{split}
& \ \int_\Sigma   N \la \bH \gamma - \bA , L_Y  \gamma  \ra \, d \sigma \\
= & \ - 2 \int_\Sigma ( \bH \gamma  - \bA) (\nabla N, Y ) - 2 \int_\Sigma N  ( d  \bH - \div \bA ) (Y) \, d \sigma .
\end{split}
\ee  
By the Codazzi equation and the static equation, 
\be
\begin{split}
N ( \div \bA - d \bH ) (Y ) =  N \bar Ric( Y, \bnu) 
=  \bar D^2 N ( Y, \bnu) .
\end{split}
\ee
Here 
$$  \bar D^2 N (Y, \nu)  =  - \bA ( \nabla N, Y ) + Y \left(  \frac{\p  N}{\p \bnu} \right) . $$
Hence, 
\be
\begin{split}
 \int_\Sigma   N \la \bH \gamma - \bA , L_Y \gamma  \ra \, d \sigma 
=  \int_\Sigma  -2 \bH \la \nabla N, Y \ra + 2  Y \left(  \frac{\p  N}{\p \bnu} \right)  \, d \sigma .
\end{split}
\ee
Therefore,  \eqref{eq-over-I} can be rewritten as 
\be
\begin{split}
& \  \frac{d}{dt} \int_\Sigma N \bH \, d \sigma \\
= & \  \int_\Sigma 2  f \frac{\p N}{\p \bnu} \bH  
+  \bH \la \nabla N, Y \ra -   Y \left(  \frac{\p  N}{\p \bnu} \right) + \frac12 N \la \bH \gamma - \bA , \gamma' \ra \, d \sigma . 
\end{split}
\ee

We now turn to the term  $\int_\Sigma N  H \, d \sigma $. 
We have
\be
\begin{split}
\frac{d}{dt} \int_\Sigma N  H \, d \sigma 
= & \ \int_\Sigma N' H + N H'   + N H \eta H \,  d \sigma \\
= & \ \int_\Sigma \left( f \frac{\p N}{\p \bnu} + \la \nabla N, Y \ra \right) H \\ 
& \ +  N \left[ - \Delta \eta  - ( | A |^2 + \Ric(\nu, \nu) ) \eta   \right]  + N H^2  \eta \, d \sigma .
\end{split}
\ee
Here 
\be
\begin{split}
- \int_\Sigma N \Delta \eta \, d \sigma =  - \int_{\Sigma} ( \Delta N) \eta  \, d \sigma 
= \int_\Sigma  \left( \bH \frac{\p N}{\p \bnu} + \bar Ric (\bnu, \bnu) N \right) \eta .
\end{split}
\ee
Therefore, 
\be
\begin{split} 
 \frac{d}{dt} \int_\Sigma N  H \, d \sigma 
= & \ 
 \int_\Sigma f \frac{\p N}{\p \bnu} H + \la \nabla N, Y \ra  H +  \bH \frac{\p N}{\p \bnu} \eta \\
& \  +  N  \left[ \bar Ric (\bnu, \bnu)  -   ( | A |^2 + \Ric(\nu, \nu) )   +  H^2 \right]  \eta \, d \sigma
\end{split}
\ee

We group the zero order terms of $N$ in $ \displaystyle \frac{d}{dt} \int_\Sigma N (\bH - H ) \, d \sigma $ 
 first. Using $ \gamma' = 2 \eta A $, we have
\be
\frac12 N \la \bH \gamma - \bA , \gamma' \ra  =  N \la \bH \gamma - \bA ,  A \ra \eta .
\ee
Thus, omitting the terms $\eta $ and $N$,  using the Gauss equation, we have
\be
\begin{split}
& \ \la \bH \gamma - \bA ,  A \ra -  \bar Ric (\bnu, \bnu)  +   \Ric(\nu, \nu) )  + | A|^2 -  H^2 \\
= & \   \la \bH \gamma - \bA ,  A \ra  + \frac12 ( R - \bar R) - \frac12 ( H^2 -  | A |^2 ) - \frac12 ( \bH^2 - | \bA |^2 )  \\
= & \   \frac12 | A - \bA |^2 - \frac12 | H -  \bH |^2  + \frac12 ( R - \bar R) .
\end{split}
\ee
Integrating by part and using the fact 
$ \eta H = f \bH + \div Y $, 
we conclude 
\be
\begin{split}
& \ \frac{d}{dt} \int_\Sigma N (\bH - H ) \, d \sigma \\
= & \ \int_\Sigma N \left[ \frac12 | A - \bA |^2 - \frac12 | H -  \bH |^2  + \frac12 ( R - \bar R)  \right] \eta \, d \sigma \\
& \ +   \int_\Sigma  ( 2  f \bH  - f H - \eta \bH + \div Y   )     \frac{\p N}{\p \bnu} 
+  ( \bH - H ) \la \nabla N, Y \ra   
  \, d \sigma \\
= & \ \int_\Sigma N \left[ \frac12 | A - \bA |^2 - \frac12 | H -  \bH |^2  + \frac12 ( R - \bar R)  \right] \eta \, d \sigma \\
& \ + \int_\Sigma   \left[ ( f  - \eta )      \frac{\p N}{\p \bnu}  + \la \nabla N, Y \ra \right]  ( \bH - H )   \, d \sigma  .
\end{split}
\ee
\end{proof}

\section{Equality case of the localized Penrose inequality} \label{sec-rigidity}

In this section, we apply Formula \ref{formula-I}, the openness result of the isometric embedding problem \cite{Li-Wang}, and Theorem \ref{thm-L-M} to prove Theorem \ref{thm-main}. 

\begin{proof}[Proof of Theorem \ref{thm-main}]
 Let $ A$, $ \bA$ be the second fundamental form of $ \So$, $ \Sigma$ in $(\Omega, \fg)$, $\StM$, respectively.
Viewing $\bA$ as a tensor on $\So$ via the surface isometry, we want to show $ A = \bA$. 

In $(\Omega, \fg)$, consider a smooth family of $2$-surfaces $\{ \Sigma_t \}_{ - \epsilon < t \le 0 }$ such that
$ \Sigma_0 = \So $ and $ \Sigma_t $ is $|t|$-distance away from $\So$. 
We can parametrize $\{ \Sigma_t \}$ so that, as $t$ increases, $ \Sigma_t $ evolves in a direction normal to $\Sigma_t$  and has constant unit speed. Applying the openness result of the isometric embedding problem in \cite{Li-Wang},  
we obtain a smooth family of $2$-surfaces $\{ \bSigma_t \}_{ - \epsilon < t \le 0 } $ in $\StM$ so that 
$ \bSigma_0 = \Sigma$ and  condition \eqref{eq-iso} is satisfied by $\{ \Sigma_t \}$ and $ \{ \bSigma_t \}$. 
By \eqref{eq-evolution-qlmass} and the assumption $ H = H_m $, we have
\be \label{eq-app-1}
\frac{d}{dt}|_{t=0} \int_{\Sigma_t} N (\bH - H ) \, d \sigma  =  \frac12 \int_{\So} N ( | A - \bA |^2 + R ) \, d \sigma .
\ee 
Here $N$ is the static potential on $\StM$, which is positive away from the horizon, 
and $ R$ is the scalar curvature of $(\Omega, \fg)$. 

Suppose  $ A \neq \bA$. Then, by \eqref{eq-app-1} and  the assumption $ R \ge 0$, 
\be
\frac{d}{dt}|_{t=0} \int_{\Sigma_t} N (\bH - H ) \, d \sigma  > 0 .
\ee
Thus, for small $ t < 0$, 
\be \label{eq-cal}
\int_{\Sigma_t} N (\bH - H ) \, d \sigma < 0 .
\ee 

We claim  \eqref{eq-cal} contradicts Theorem \ref{ineq-main}.
To see this, we can first consider the case   $ \bRic (\nu, \nu) < 0 $ on $ \Sigma $. 
By choosing $ \epsilon $ small, we may assume 
$ \bRic(\nu, \nu) < 0 $ on each $ \bSigma_t$. 
Hence, we can apply Theorem \ref{thm-L-M} to the region in $\Omega$ 
enclosed by $ \Sigma_t$ and $ \Sh$. 
It  follows from  \eqref{ineq-main} 
and the assumption $ m = \sqrt{ \frac{ | \Sh | }{ 16\pi} }$ that
\be \label{eq-app-ineq}
\int_{\Sigma_t} N (\bH - H ) \, d \sigma \ge 0 .
\ee 
This is a  contradiction to  \eqref{eq-cal}. 

To include the case $ \bRic(\nu, \nu) \le 0 $ on $ \Sigma$,
we point out that this assumption was imposed in \cite{Lu-Miao}
only to guarantee that  the flow in $\StM$, which starts from $ \Sigma$ 
and satisfies equation (4.2) in \cite{Lu-Miao},
has the property that its leaves have positive scalar curvature (see Lemma 3.8 in \cite{Lu-Miao}).
Now, if $ \Sigma$ is slightly perturbed to a nearby surface $\Sigma'$ in $ \StM$, 
though $ \Sigma'$ may not satisfy $ \bRic(\nu,\nu) \le 0 $, 
the flow to (4.2) in \cite{Lu-Miao} starting from $ \Sigma'$ 
remains to have such a property.
(More precisely, this follows from estimates in Lemmas 3.6, 3.7 and 3.11 of \cite{Lu-Miao}.)
Therefore, for small $ t < 0 $, we can still apply Theorem \ref{thm-L-M} to conclude  \eqref{eq-app-ineq},
which contradicts \eqref{eq-cal}.

Thus we have  $ A = \bA$. For the same reason, we also know $ R = 0 $ along $ \So$ in $(\Omega, \fg)$.
Next, we consider the manifold $(\hat M, \hat g)$ obtained by gluing $(\Omega, \fg)$ 
and $ (\StM \setminus \Omega_m , \bg) $ along $ \So $ that  is identified with $ \Sigma$. 
Since $ A = \bar A$, the metric $ \hat g$ on $\hat M$ is $C^{1,1}$ across $ \So$ and is smooth up to $ \So$ 
from its both sides in $\hat M$. 
To finish the proof,  we check that the rigidity statement of the Riemannian Penrose inequality 
holds on this $(\hat M, \hat g)$. 

We apply the conformal flow used by Bray \cite{Bray} in his proof of 
the Riemannian Penrose inequality.  
Since $\hat g$ is $C^{1,1}$, equations (13) - (16) in \cite{Bray} which define  the flow hold in the classical sense when $g_0$ is replaced by $\hat g$.
Existence of this flow with initial condition $\hat g$ follows from Section 4 in \cite{Bray}. 
The  difference is that,  along the flow which we denote by $ \{ \hat g (t)\}$, 
the outer minimizing horizon $\Sigma(t)$ is $C^{2,\alpha}$ 
and the green function in Theorems 8 and 9 in \cite{Bray} is $C^{2,\alpha}$, for any $0<\alpha<1$.
These regularities are sufficient to show Theorem 6 in \cite{Bray} holds, i.e. the area of 
$\Sigma (t)$  stays the same; and  the results on the mass and the capacity in 
Theorems 8 and 9 in \cite{Bray} remain valid. 
Moreover, at $ t = 0 $, by the proof of Theorem 10 in \cite{Bray}, i.e. equation (113), we have
\be \label{eq-mass-c}
\frac{d}{dt^{+}}m(t)|_{t=0}=\mathcal{E}(\Sh,  \hat g)-2m\leq 0,
\ee
where $ \mathcal{E}(\Sh, \hat g)$  is the capacity of $ \Sh $ in $(\hat M, \hat g)$
and the inequality in \eqref{eq-mass-c} is given by Theorem 9 in \cite{Bray}.

Now, if  $\frac{d}{dt^{+}}m(t)|_{t=0}<0$, then for $t$ small, we would have 
\be  \label{eq-mass-area}
m (t) < m = \sqrt{\frac{|\Sh |}{16\pi}} =  \sqrt{\frac{|\Sigma (t) |}{16\pi}}, 
\ee
where $m(t)$ is the mass of  $\hat g(t)$. 
But \eqref{eq-mass-area} violates the Riemannian Penrose inequality
(for metrics possibly with corner along a hypersurface, cf. \cite{MM17}).
Thus, we must have 
\begin{align*}
\frac{d}{dt^{+}}m(t)|_{t=0}=\mathcal{E}(\Sh, \hat g )-2m= 0.
\end{align*}
Since Theorem 9 in \cite{Bray} holds on $(\hat M, \hat g)$, 
by its rigidity statement we conclude that $(\hat M, \hat g)$ is isometric to $\StM$. 
\end{proof}

\begin{rem}
As mentioned  in \cite[Remark 5.1]{Lu-Miao}, 
Theorem \ref{thm-main}  would also follow
if one could establish the rigidity statement for the Riemannian Penrose inequality on manifolds with corners along a hypersurface (cf. \cite[Proposition 3.1]{MM17}). 
Results along this direction can be found in \cite{ShiWangYu}.
\end{rem}

\section{Bartnik mass and Wang-Yau quasi-local energy}   
\label{sec-discussion}

In \eqref{eq-evolution-qlnergy-mb} of Section \ref{sec-evolution-qlmass},
we have observed that, if $(\N, \bg)$ represents a mass minimizing, static extension of
$\Sigma_0 \subset (M, g)$, then 
\be \label{eq-dev-bm-ql}
 \frac{d}{dt}|_{t=0}  \ESwy (\Sigma_t, \S, X_t) 
= \frac{d}{dt}|_{t=0}  \mathfrak{m}_{_B} (\Sigma_t) .
\ee
This observation was based on \eqref{eq-evolution-bm}, which requires 
a  stringent assumption that mass minimizing, static extensions of $ \{ \Sigma_t \} $ exist
and depend smoothly on $t$. 
In this section, we will give a rigorous proof  
 that \eqref{eq-evolution-bm}  is true
whenever  the Bartnik data of $\Sigma_0$ 
corresponds to that of  a surface in a spatial Schwarzschild manifold. 
We will also discuss  the implication, suggested by \eqref{eq-dev-bm-ql},
 on the relation between the Bartnik mass and the  Wang-Yau quasi-local energy.

First, we recall the definition of $\mb (\cdot)$.
Given a closed $2$-surface $\Sigma $, which bounds a bounded domain, in 
a $3$-manifold $(M, g)$ with nonnegative scalar curvature, 
$ \mb(\Sigma)$ is given by 
\be \label{eq-def-bmass}
\mB  (\Sigma)  = \inf \left\{  \m (\tg ) \, \vert \, (\tilde M, \tilde g ) \text{\ is  an  admissible extension of }
\Sigma \right\}.
\ee
Here $\m ( \tg )$ is the mass of $ (\tM, \tg)$, which is an asymptotically flat  $3$-manifold with nonnegative 
scalar curvature, with boundary $\p \tM$.   
$(\tM, \tg)$  is called an admissible extension 
 of $\Sigma$ if  $ \p \tM$  is isometric to 
$\Sigma$  and the mean curvature of $ \p \tM $  equals the mean curvature  $H$
of $ \Sigma$.
Moreover, it is assumed that   $(\tM, \tg)$ satisfies certain non-degeneracy condition 
that prevents $\m (\tg)$ from becoming trivially small. For instance, one often assumes that 
$(\tM, \tg)$  contains no closed minimal surfaces or  $\p \tM$ is outer minimizing in $(\tM, \tg)$ 
(cf. \cite{Bartnik, Bray, B04, HI01}). 

\begin{theo} \label{thm-bm-dev-S}
Let $\Sigma$ be a $2$-surface with positive mean curvature  in a 
$3$-manifold  $(M, g)$ of nonnegative scalar curvature. 
Suppose $\Sigma $ is isometric to a convex surface $ \bSigma$
 with $ \bRic(\nu, \nu) \le 0 $  in a spatial Schwarzschild manifold $(\StM, \bg)$
 of mass $ m > 0 $. 
Suppose $\bSigma$ encloses a domain $\Omega_m$ with the horizon 
of $(\StM, \bg)$.
\begin{enumerate}
\item[(i)] Let $ X: \Sigma \rightarrow (\StM, \bg) $ be an isometric embedding 
such that $ X (\Sigma) = \bSigma$.  Let 
$N$ be the static potential on $ \StM$ and let $ \S_m$ denote the Schwarzschild spacetime, 
i.e. $ \S_m = ( \R^1 \times \StM, - N dt^2 + \bg) $.
Then
\bee
\mb (\Sigma) \le m + \ESwy (\Sigma, \S_m, X) .
\eee
Moreover, equality holds if and only if $ H  = \bar H $ and $ \mb(\Sigma) = m $.
Here $H$, $ \bH $ are the mean curvature of $\Sigma$, $ \bSigma$ in  $(M, g)$, $ (\StM, \bg)$,
respectively. 

\item[(ii)] Suppose $H = \bar H $. 
Let  $\{ \Sigma_t \}_{ | t | < \epsilon } $ be a smooth family of $2$-surfaces evolving in $(M,g)$
according to $ \frac{\p F}{\p t} = \eta \nu$ and satisfying $\Sigma_0 = \Sigma$.
If  $\mb(\Sigma_t)$ is differentiable at $t=0$, then
\bee \label{eq-bm-dev-s}
\frac{d}{dt}|_{t=0}  \mathfrak{m}_{_B} (\Sigma_t)  = 
\frac{1}{16 \pi} \int_{\Sigma_0} N ( | A - \bA |^2 + R ) \eta \, d \sigma. 
\eee
Here  $ A$, $ \bA$ are the second fundamental form of $\Sigma$, $ \bSigma$ in
$(M, g)$, $(\StM, \bg)$, respectively, and $ R$ is the scalar curvature of $(M,g)$.
\end{enumerate}
\end{theo}

\begin{proof} Part (i) was proved in \cite[Theorem 5.1]{Lu-Miao}.
To show part (ii),  we first note that the assumption $H = \bH$ implies   $ \mb (\Sigma ) = m $. 
This is because, if $(\tM, \tg)$ is any  admissible extension of $\Sigma $, 
by gluing $(\tM, \tg)$ with $\Omega_m$ along $ \bSigma_0$ and applying the 
Riemannian Penrose inequality, one  has $ \m (\tg) \ge m $. 
On the other hand, $\StM \setminus \Omega_m$ is an admissible extension of $\Sigma $.
Hence, $ \mb (\Sigma ) = m $. 

Next, we proceed as in the proof of Theorem \ref{thm-main}.
By the result of Li-Wang \cite{Li-Wang}, for small $ \epsilon $, there exists 
a smooth family of embeddings $\{ X_t \}_{|t| < \epsilon}$
which isometrically embeds $\Sigma_t$ in $(\StM, \bg)$ such that $X_0 = X$. 
By (i), for each small $t$, we have
\be \label{eq-bm-ql-c}
\mb(\Sigma_t) \le m + \ESwy(\Sigma_t, \S_m, X_t) .
\ee
Note that  $ \mb(\Sigma_0) = m $ and  $ \ESwy(\Sigma_0, \S_m, X_0) = 0 $.
Hence, it follows from \eqref{eq-bm-ql-c} that   
\be \label{eq-two-dev-s}
\begin{split}
\frac{d}{dt}|_{t=0}  \mb(\Sigma_t) = & \  
\frac{d}{dt}|_{t=0}   \ESwy(\Sigma_t, \S_m, X_t) \\
= & \ \frac{1}{16 \pi} \int_{\Sigma_0} N ( | A - \bA |^2 + R ) \eta \, d \sigma .
\end{split}
\ee
Here in the last step we have used  \eqref{eq-evolution-qlmass}.
\end{proof}

We propose a conjecture that is inspired by Theorem \ref{thm-bm-dev-S}.

\begin{conj}  \label{conj-bm-s}
Let $ \Sigma $ be a $2$-surface, bounding some finite domain, 
 in a $3$-manifold $(M, g)$ of nonnegative scalar curvature.
Let $(\N, \bg)$ be a complete, asymptotically flat $3$-manifold  with nonnegative 
scalar curvature such that $(\N, \bg)$ is static outside a compact set $K$. 
Suppose the static potential $N$ on $\N \setminus K$ is positive.  
Let $ \S$ be the static spacetime generated by $ (\N \setminus K, \bg)$, i.e.
$ \S = ( \R^1 \times ( \N \setminus K) , - N^2 d t^2 + \bg)$.
Suppose there exists an isometric embedding $X : \Sigma \rightarrow (\N , \bg) $
such that $ \bSigma= X(\Sigma)$ encloses $K$. 
 Let $ H$, $ \bH$ be the mean curvature of $\Sigma$, $\bSigma$
 in $(M, g)$, $(\N, \bg)$, respectively. 
Then, under suitable conditions on $\Sigma$ and $ \bSigma $, 
\be \label{eq-bm-s}
 \mb(\Sigma) \le \m (\bg)  + \ESwy (\Sigma, \S, X) .
\ee
Moreover, equality holds if and only if  
$$ H = \bH \ \ \mathrm{and} \ \ \mb(\Sigma) = \m(\bg) , $$
 in which case $(\N, \bg)$, outside $\bSigma$,  is 
a mass minimizing, static extension of $\Sigma$. 
\end{conj}

By results in \cite{Shi-Tam02}, 
Conjecture \ref{conj-bm-s} is true when $(\N, \bg)$ is $\R^3$. 
By Theorem \ref{thm-bm-dev-S} (i), 
Conjecture \ref{conj-bm-s} is  also true when $(\N \setminus K, \bg)$ 
is an exterior region in $(\StM, \bg)$.

If Conjecture \ref{conj-bm-s} is valid and if a mass minimizing, static extension of $\Sigma$ exists, 
then it would follow that 
\bee
 \mb(\Sigma) = \inf_{(\N, \bg)} \left\{ \inf_{X}  \left(  \m (\bg)  + \ESwy (\Sigma, \S, X) \right)  \right\}.
\eee

\vspace{.2cm}

\noindent {\bf Acknowledgement}. The authors would like to thank Po-Ning Chen and Xiangwen Zhang for
helpful comments on this work.


\begin{thebibliography}{10}

\bibitem{Bartnik} 
R. Bartnik,
{\em  New definition of quasilocal mass}, 
Phys. Rev. Lett. \textbf{62} (1989),  2346--2348.

\bibitem{Bray} 
H. L. Bray, 
{\em Proof of the Riemannian Penrose inequality using the positive mass theorem}.
J. Differential Geom. \textbf{59} (2001), no. 2, 177--267. 

\bibitem{B04} 
H. L. Bray and P. T. Chrus\'ciel, {\em The Penrose inequality}, 
in: The Einstein Equations and the Large Scale Behavior of Gravitational Fields, 
Eds P. T. Chrus\'ciel and H. Friedrich, Birkh\"auser Verlag, Basel, (2004), 39--70.

\bibitem{CWWY}
P.-N. Chen, M.-T. Wang, Y.-K. Wang and S.-T. Yau, 
{\em Quasi-local energy with respect to a static spacetime},
arXiv:1604.02983.

\bibitem{Chen-Zhang}
P.-N. Chen and  X. Zhang, {\em A rigidity theorem for surfaces in Schwarzschild manifold}, arXiv:1802.00887.

\bibitem{Corvino} 
J. Corvino, {\em Scalar curvature deformation and a gluing construction for the Einstein constraint equations}, 
Comm. Math. Phys. \textbf{214} (2000), 137--189.

\bibitem{HI01} 
G. Huisken and T. Ilmanen,
{\em The inverse mean curvature flow and the {R}iemannian {P}enrose inequality},
J. Differential  Geom. \textbf{59} (2001), no. 3, 353--437.


\bibitem{Li-Wang}  C. Li and Z. Wang, {\em The Weyl problem in warped product spaces}, arXiv:1603.01350.

\bibitem{Li-Wang-more}  C. Li and Z. Wang, private communication.

\bibitem{Lu-Miao} S. Lu and P. Miao, {\em Minimal hypersurfaces and boundary behavior of compact 
manifolds with nonnegative scalar curvature}, arXiv: 1703.08164. 

\bibitem{MM17}
S. McCormick and P. Miao, {\em On a Penrose-like inequality in dimensions less than eight}, 
Int. Math. Res. Not. IMRN, rnx181, https://doi.org/10.1093/imrn/rnx181.

\bibitem{Miao-ICCM-07}
P. Miao,  {\em Some recent developments of the Bartnik mass}, 
Proceedings of the 4th International Congress of Chinese Mathematicians, Vol. III, 331--340, High Education Press, 2007. 

\bibitem{Shi-Tam02} 
Y.-G.  Shi and  L.-F.  Tam, {\em Positive mass theorem and the boundary behaviors of compact manifolds with nonnegative scalar curvature},  J. Differential Geom. \textbf{62} (2002), 79--125.


\bibitem{ShiWangYu} 
Y.-G.  Shi, W.-L Wang, and H.-B, Yu, 
{\em On the rigidity of Riemannian-Penrose inequality for asymptotically flat 3-manifolds with corners}, 
arXiv:1708.06373.

\bibitem{WY1} M.-T. Wang  and  S.-T. Yau,  
{\em Quasilocal mass in general relativity},  Phys. Rev. Lett. \textbf{102} (2009), no. 2, no. 021101, 4 pp.

\bibitem{WY2}  M.-T. Wang  and  S.-T. Yau, 
{\em Isometric embeddings into the Minkowski space and new quasi-local mass}, 
Comm. Math. Phys. \textbf{288}  (2009), no. 3, 919-942. 

\end{thebibliography}
\end{document}